\documentclass[12pt]{amsart}

\usepackage{amssymb,latexsym}

\makeatletter

\@namedef{subjclassname@2010}{
\textup{2010} Mathematics Subject Classification}
\usepackage{epstopdf}
\makeatother

\newtheorem{thm}{Theorem}

\newtheorem{lem}{Lemma}
\theoremstyle{definition}

\numberwithin{equation}{section}
\newcommand{\bd}{{\rm bd}}
\newcommand{\inter}{{\rm int}}
\newcommand{\diam}{{\rm diam}}
\newcommand{\width}{\rm width}

\begin{document}

\baselineskip=16pt

\title[COMPLETE SPHERICAL CONVEX BODIES]{COMPLETE SPHERICAL CONVEX BODIES}

\author{Marek Lassak}

\date{}
 
\maketitle

\thispagestyle{empty}

\begin{abstract}
Similarly to the classic notion in $E^d$, a subset of a positive diameter below $\frac{\pi}{2}$ of a hemisphere of the sphere $S^d$ is called complete, provided adding any extra point increases its diameter.
Complete sets are convex bodies on $S^d$.
Our main theorem says that on $S^d$ complete bodies of diameter $\delta$ coincide with bodies of constant width $\delta$.
\end{abstract}

\maketitle

\section {On spherical geometry}
\label{spherical}

Let $S^d$ be the unit sphere in the $(d+1)$-dimensional Euclidean space $E^{d+1}$, where $d\geq 2$.
By a {\it great circle} of $S^d$ we mean the intersection of $S^d$ with any two-dimensional subspace of $E^{d+1}$.
The common part of the sphere $S^d$ with any hyper-subspace of $E^{d+1}$ is called a {\it $(d-1)$-dimensional great sphere} of $S^d$.
By a pair of {\it antipodes} of $S^d$ we mean any pair of points of intersection of $S^d$ with a straight line through the origin of $E^{d+1}$.

Clearly, if two different points $a, b \in S^d$ are not antipodes, there is exactly one great circle containing them.
By the {\it arc} $ab$ connecting $a$ with $b$ we mean the shorter part of the great circle containing $a$ and $b$.
By the {\it spherical distance} $|ab|$, or shortly {\it distance}, of these points we mean the length of the arc connecting them.
The {\it diameter} $\diam(A)$ of a set $A \subset S^d$ is the number ${\rm sup}_{a, b \in A} |ab|$.
By a {\it spherical ball $B_{\rho}(r)$ of radius $\rho \in (0, {\pi \over 2}]$}, or shorter {\it a ball}, we mean the set of points of $S^d$ having distance at most $\rho$ from a fixed point, called the {\it center} of this ball.
Spherical balls of radius $\pi \over 2$ are called {\it hemispheres}.
Two hemispheres whose centers are antipodes are called {\it opposite hemispheres}.

We say that a subset of $S^d$ is {\it convex} if it does not contain any pair of antipodes and if together with every two points $a, b$ it contains the arc $ab$.
By a {\it convex body}, or shortly {\it body}, on $S^d$ we mean any closed convex set with non-empty interior.

Recall a few notions from \cite{L1-AEQ}.
If for a hemisphere $H$ containing a convex body $C \subset S^d$ we have $\bd(H) \cap C \not = \emptyset$, then we say that $H$ {\it supports} $C$.
If hemispheres $G$ and $H$ of $S^d$ are different and not opposite, then $L = G \cap H$ is called {\it a lune} of $S^d$.
The $(d-1)$-dimensional hemispheres bounding the lune $L$ and contained in $G$ and $H$, respectively, are denoted by $G/H$  and $H/G$.
We define the {\it thickness of a lune} $L = G \cap H$ as the spherical distance of the centers of $G/H$ and $H/G$.
For a hemisphere $H$ supporting a convex body $C \subset S^d$ we define the {\it width $\width_H(C)$ of $C$ determined by} $H$ as the minimum thickness of a lune of the form $H \cap H'$, where $H'$ is a hemisphere, containing $C$.
If for all hemispheres $H$ supporting $C$ we have $\width_H(C) = w$, we say that $C$ is {\it of constant width} $w$.

\section {Spherical complete bodies}
\label{lemmas}

Similarly to the traditional notion of a complete set in the Euclidean space $E^d$ (for instance, see \cite{BF}, \cite{CS}, \cite{Eg} and \cite{MMO}) we say that a subset $K$ of diameter $\delta \in (0, \frac{\pi}{2})$
of a hemisphere of $S^d$ is {\it complete} provided  $\diam (K \cup \{x\}) > \delta$ for every $x \not \in K.$ [This definition corrects the definition given in the version v1 and in the version published on-line in Journal of Geometry on July 4, 2020.]

\begin{thm}
An arbitrary subset of diameter $\delta \in (0, \frac{\pi}{2})$ of a hemisphere of $S^d$ is a subset of a complete set of diameter $\delta$. 
\end{thm}

We omit the proof since it is similar to the proof by Lebesgue \cite{Le} in $E^d$ (it is recalled in Part 64 of \cite{BF}).
Let us add that earlier P\'al \cite{Pa} proved this for $E^2$ by a different method.

The following fact permits to use the term  a {\it complete convex body} for a complete set.

\begin{lem}\label{intersection}
Let $K \subset S^d$ be a complete set of diameter $\delta$.
Then $K$ coincides with the intersection of all balls of radius $\delta$ centered at points of $K$.
Moreover, $K$ is a convex body.
\end{lem}

\begin{proof}
Denote by $I$ the intersection of all balls of radius $\delta$ with centers in $K$.

Since $\diam (K)  = \delta$, then $K$ is contained in every ball of radius $\delta$ whose center is a point of $K$.
Consequently,  $K \subset I$.

Let us show that $I \subset K$; so let us show that $x \not \in K$ implies $x \not \in I$.
Really, from $x \not \in K$ we get $|xy| > \delta$ for a point $y \in K$, which means that $x$ is not in the ball of radius $\delta$ with center $y$, and thus $x \not \in I$.

As an intersection of balls centered at points of $K$, our $K$ is a convex body.
\end{proof}

\begin{lem}\label{p'}
If $K \subset S^d$ is a complete body of diameter $\delta$, then for every $p \in \bd (K)$ there exists $p' \in K$ such that $|pp'|=\delta$.
\end{lem}

\begin{proof}
Suppose the contrary, i.e., that $|pq| < \delta$ for a point $p \in \bd (K)$ and for every point $q \in K$.
Since $K$ is compact, there is an $\varepsilon > 0$ such that  $|pq| \leq \delta - \varepsilon$ for every $q \in K$.
Hence there is a point $s \not \in K$ in a positive distance from $p$ which is smaller than $\varepsilon$ such that $|sq| \leq \delta$ for every $q \in K$.
Thus $\diam (K \cup \{s\}) = \delta$, which contradicts the assumption that $K$ is complete.
Consequently, the thesis of our lemma holds true.
\end{proof}

For different points $a, b \in S^d$ at a distance $\delta < \pi$ from a point $c \in S^d$ define the {\it piece of circle} $P_c(a,b)$ as the set of points $v \in S^d$ such that $cv$ has length $\delta$ and intersects $ab$.

We show the next lemma for $S^d$ despite we apply it later
only for $S^2$.

\begin{lem}\label{circle}
Let $K \subset S^d$ be a complete convex body of diameter $\delta$.
Take $P_c(a,b)$ with $|ac|$ and $|bc|$ equal to $\delta$ such that $a, b \in K$ and $c \in S^d$.
Then $P_c(a,b) \subset K$.
\end{lem}

\begin{proof}
First let us show the thesis for a ball $B$ of radius $\delta$ in place of $K$.
There is unique $S^2 \subset S^d$ with $a, b, c \in S^2$.
Consider the disk $D= B \cap S^2$.
Take the great circle containing $P_c(a,b)$ and points $a^*, b^*$ of its intersection with the circle bounding $D$.
There is a unique $c^* \in S^2$ such that $P_c(a,b) \subset P_{c^*}(a^*,b^*)$.
Clearly, $P_{c^*}(a^*,b^*) \subset D \subset B$.
Hence $P_c(a,b) \subset B.$

By the preceding paragraph and Lemma \ref{intersection} we obtain the thesis of the present lemma.
\end{proof}

\section{Complete and constant width bodies on $S^d$ coincide}
\label{coincide}

Here is our main result presenting
the spherical version of the classic theorem in $E^d$ proved by Meissner \cite {Me} for $d=2,3$ and by Jessen \cite{Je}
for arbitrary $d$.

\begin{thm}\label {main}
A body of diameter $\delta$ on $S^d$ is complete if and only if it is of constant width~$\delta$.
\end{thm}

\begin{proof}
($\Rightarrow$)
Let us prove that if a body $K \subset S^d$ of diameter $\delta$ is complete, then $K$ is of constant width $\delta$.

Suppose the opposite, i.e., that $\width_I(K) \not = \delta$ for a hemisphere $I$ supporting $K$.
By Theorem 3 and Proposition 1 of \cite{L1-AEQ} $\width_I(K) \leq \delta$.
So $\Delta (K) < \delta$.
By lines 1-2 of p. 562 of \cite{L1-AEQ} the thickness of $K$ is equal to the minimum thickness of a lune containing $K$.
Take such a lune $L = G \cap H$, where $G, H$ are different and non-opposite hemispheres.
Denote by $g, h$ the centers of $G/H$ and $H/G$, respectively.
Of course, $|gh| < \delta$.
By Claim 2 of \cite{L1-AEQ} we have $g, h \in K$.
By Lemma \ref{p'} there exists a point $g' \in K$ in the distance $\delta$ from $g$.
Since the triangle $ghg'$ is non-degenerate, there is a unique two-dimensional sphere $S^2 \subset S^d$ containing $g, h, g'$.
Clearly, $ghg'$ is a subset of $M= K \cap S^2$.
Hence $M$ is a convex body on $S^2$.
Denote by $F$ this hemisphere of $S^2$ such that $hg' \subset \bd(F)$ and $g \in F$.
There is a unique $c \in F$ such that $|ch| = \delta = |cg'|$.
By Lemma \ref{circle} for $d=2$ we have $P_c(h,g') \subset M$.

We intend to show that $c$ is not on the great circle $E$ of $S^2$ through $g$ and $h$.
In order to see this, for a while suppose the opposite, i.e. that $c \in E$.
Then from $|g'g| = \delta$, $|g'c| = \delta$ and $|hc| = \delta$ we conclude  that $\angle gg'c = \angle hcg'$.
So the spherical triangle $g'gc$ is isosceles, which together with $|gg'|= \delta$ gives $|cg| = \delta$.
Since $|gh| = \Delta(L) = \Delta(K) >0$ and $g$ is a point of $ch$ different from $c$, we get a contradiction.
Hence, really, $c \not \in E$.

By the preceding paragraph $P_c(h,g')$ intersects $\bd(M)$ at a point $h'$ different from $h$ and $g'$.
So the non-empty set $P_c(h,g') \setminus \{h,h'\}$ is out of $M$.
This contradicts the result of the paragraph before the last.
Consequently, $K$ is a body of constant width $\delta$.

($\Leftarrow$) Let us prove that if $K$ is a spherical body of constant width $\delta$, then $K$ is a complete body of diameter $\delta$.
In order to prove this, it is sufficient to take any point $r \not \in K$ and to show that $\diam (K \cup \{r\}) > \delta$.

Take the largest ball $B_{\rho}(r)$ disjoint with the interior of $K$.
Since $K$ is convex, $B_{\rho}(r)$ has exactly one point $p$ in common with $K$.
By Theorem 3 of \cite{LaMu-AEQ} there exists a lune $L \supset K$
of thickness $\delta$ with $p$ as the center of one of the two $(d-1)$-dimensional hemispheres bounding this lune.
Denote by $q$ the center of the other $(d-1)$-dimensional hemisphere.
By Claim 2 of \cite{L1-AEQ} also $q \in K$.
Since $p$ and $q$ are the centers of the two $(d-1)$-dimensional hemispheres bounding $L$, we have $|pq|= \delta$.
From the fact that $rp$ and $pq$ are orthogonal to $\bd (H)$ at $p$,
we see that $p \in rq$.
Moreover, $p$ is not an endpoint of $rq$ and $|pq| = \delta$,
Hence $|rq| > \delta$.
Thus $\diam (K \cup \{r\}) > \delta$.
Since $r \not \in K$ is arbitrary, $K$ is complete.
\end{proof}

We say that a convex body $D \subset S^d$ is of {\it constant diameter} $\delta$ provided $\diam(D) = \delta$ and for every $p \in \bd(D)$ there is a point $p' \in \bd(D)$ with $|pp'| = \delta$ (see \cite {LaMu-AEQ}).
The following fact is analogous to the result in $E^d$ given by Reidemeister \cite {Re}.

\begin{thm}\label{constant diameter}
Bodies of constant diameter on $S^d$ coincide with complete bodies.
\end{thm}

\begin{proof}
Take a complete body $D \subset S^d$ of diameter $\delta$.
Let $g \in \bd (D)$ and $G$ be a hemisphere supporting $D$ at $g$.
By Theorem \ref{main} the body $D$ is of constant width $\delta$.
So $\width_G(D) = \delta$ and there exists a hemisphere $H$ such that the lune $G \cap H \supset D$ has thickness $\delta$.
By Claim 2 of \cite{L1-AEQ} the centers $h$ of $H/G$ and $g$ of $G/H$ belong to $D$.
So $|gh| =\delta$.
Thus $D$ is of constant diameter $\delta$.

Consider a body $D \subset S^d$ of constant diameter $\delta$.
Let $r \not \in D$.
Take the largest $B_{\rho}(r)$ whose interior is disjoint with $D$.
Denote by $p$ the common point of $B_{\rho}(r)$ and $D$.
A unique hemisphere $J$ supports $B_{\rho}(r)$ at $p$.
Observe that $D \subset J$ (if not, there is a point $v \in D$ out of $J$; clearly $vp$ passes through $\inter B_{\rho}(r)$, a contradiction).
Since $D$ is of constant diameter $\delta$, there is $p' \in D$ with $|pp'| =\delta$.
Observe that $\angle rpp' \geq \frac{\pi}{2}$.
If it is $\frac{\pi}{2}$, then $|rp'| > \delta$.
If it is larger than $\frac{\pi}{2}$, the triangle $rpp'$ is obtuse and then by the law of cosines $|rp'| > |pp'|$ and hence
$|rp'| > \delta$.
By $|rp'| > \delta$ in both cases we see that $D$ is complete.
\end{proof}

By Theorem \ref{main}, in Theorem \ref{constant diameter} we may exchange ``complete" to ``constant width". 
This form is proved earlier as follows.
Any body of constant width $\delta$ is of constant diameter $\delta$ and the inverse is shown for $\delta \geq \frac{\pi}{2}$, and for $\delta < \frac{\pi}{2}$ if $d=2$ (see \cite{LaMu-AEQ}).
By \cite{HanWu} this inverse holds for any $\delta$.
Our short proof of Theorem \ref{constant diameter} is 
quite different from these in \cite{LaMu-AEQ}, \cite{L2-AEQ} and~\cite{HanWu}.

\vskip0.2cm
\baselineskip=12pt

Marek Lassak

University of Technology and Life Sciences 

al. Kaliskiego 7, Bydgoszcz 85-796 Bydgoszcz, Poland

\vskip0.1cm
e-mail: lassak@utp.edu.pl

\end{document}